\documentclass[reqno,a4paper]{amsart}
\usepackage{Bau36}

\title[Adaptive Fixed Point Iterations]{Adaptive Fixed Point Iterations For Semilinear Elliptic Partial Differential Equations}
	
\author[M.~Amrein]{Mario Amrein}
\address{Applied University of Z\"urich, CH-8004 Switzerland}
\email{mario.amrein@hslu.ch}

\begin{document}
\normalem
\begin{abstract}
In this paper we study the behavior of finite dimensional fixed point iterations, induced by discretization of a continuous fixed point iteration defined within a Banach space setting. We show that the difference between the discrete sequence and its continuous analogue can be bounded in terms depending on the mesh size of the discretization and the contraction factor, defined by the continuous iteration. Furthermore, we show that the comparison between the finite dimensional and the continuous fixed point iteration naturally paves the way towards a general \emph{a posteriori} error analysis that can be used within the framework of a fully adaptive solution procedure. In order to demonstrate our approach, we use the Galerkin approximation of singularly perturbed semilinear monotone problems. Our scheme combines the fixed point iteration with an adaptive finite element discretization procedure (based on a robust \emph{a posteriori} error analysis), thereby leading to a fully adaptive Fixed-Point-Galerkin scheme. Numerical experiments underline the robustness and reliability of the proposed approach.
\end{abstract}



\maketitle

\section{Introduction}

In this work we study the numerical approximation of problems given by: 
\begin{equation}\label{eq:problem}
\text{find} \quad u \in X: \qquad \F(u)=0, \quad \text{in} \quad X',
\end{equation}

where $\F:X \rightarrow X'$ signifies a possibly nonlinear operator. 
Here, $X$ stands for a real Hilbert space, with inner product denoted by $(\cdot,\cdot)_{X}$ and induced norm $\norm{x}_{X}=~\sqrt{(x,x)_{X}}$. Furthermore, $X'$ signifies the dual of $X$.

\subsubsection*{Fixed Point Galerkin Methods}
As a result of the possible nonlinearity of $\F$, for a given $u^{0}\in X$, we consider the fixed point iteration
\begin{equation}\label{eq:continuous-iteration}
(u^{n+1},v)_{X}=\B(u^{n})(v), \quad \forall v \in X,
\end{equation}
where, for a fixed $t>0$, the operator $\B$ is defined by
\begin{equation}\label{eq:continuous}
\B(u)(v):=(u,v)_{X}- t\cdot \F(u)(v), \qquad \forall v\in X.
\end{equation}

To guarantee the convergence of the above iteration, we suppose that $\F$ satisfies
the following two assumptions, namely the Lipschitz continuity property
\begin{equation}
\label{eq:A1}\tag{A1}
\abs{\F(x)(v)-\F(y)(v)}\leq L\norm{x-y}_{X}\norm{v}_X,\qquad L>0,
\end{equation}
as well as the strong monotonicity property
\begin{equation}\label{eq:A2}\tag{A2}
(\F(x)-\F(y))(x-y)\geq c\norm{x-y}_{X}^2, \qquad c>0.
\end{equation}

Under these assumptions, it is well known that there exists a unique solution $u \in X$ solving \eqref{eq:problem} (see, e.g. \cite{12}, or section \ref{sec:errors} in this work). More precisely, it can be shown that the operator defined in \eqref{eq:continuous} is contractive, i.e., there holds  
\begin{equation}\label{eq:contraction}
\abs{\B(x)(v)-\B(y)(v)}\leq \alpha \norm{x-y}_{X}\norm{v}_{X}, \quad \alpha= \sqrt{1-\left(\frac{c}{L}\right)^2}\in (0,1),
\end{equation}
and therefore, based on Banach's fixed point Theorem, the solution can be obtained by iterating \eqref{eq:continuous-iteration} with optimal step size $t_{\text{opt}}=\frac{c}{L^2}$ (see also \cite{CongreveWihler:15}).  

However in actual computations, we can only solve a finite dimensional analogue of equation \eqref{eq:continuous-iteration}.
More precisely, let $X_{h}\subset X$ be a linear finite dimensional subspace of $X$. 
We then observe
\[
\abs{\B(x_h)(v_h)-\B(y_h))(v_{h})}\leq \alpha_{h} \norm{x_{h}-y_{h}}_{X}\norm{v_{h}}_{X},
\]
with $0<\alpha_{h}\leq \alpha $ from where we get the existence of a unique $x_{h}\in X_{h}$ such that $ \B(x_h)(v_h)=(x_h,v_{h})_{X} \ \forall v_{h}~\in~X_{h}$ and solving 
\[
\F(x_h)(v_{h})=0, \quad \forall v_{h}\in X_{h}.
\]	

For a given initial value $u_{0}^{h}\in X_{h}$, the solution can be obtained by the \emph{fixed point iteration} 

\begin{equation}
\label{eq:discrete_iter}
(u_{h}^{n+1},v_{h})_{X}=\B(u_{h}^{n})(v_{h}), \quad \forall v_{h}\in X_{h}.
\end{equation}

In order to establish a possible \emph{a priori} error analysis, we will focus (following the argument from~\cite[\S 8.1]{5} for Newton's method) on the 
distance, between the sequence $\{u^{n}\}_{n\geq 0}$, corresponding to the iteration from \eqref{eq:continuous-iteration}, and its discrete analogue $\{u_h^{n}\}_{n\geq 0}$ defined in \eqref{eq:discrete_iter}.
More precisely: we will show that the discrete sequence tracks its continuous analogue with a maximal distance, which can be bounded in terms depending on the mesh size parameter $h>0$ and the contraction factor $\alpha \in (0,1) $ defined in \eqref{eq:contraction}.  
As it turns out, the above outlined approach naturally paves the way towards an \emph{a posteriori} error analysis, where two different error indicators contribute to the \emph{a posteriori} error bound. 
They are caused by the nonlinearity of the problem and its discretization.

Based on these two error contributions, and following along the lines of \cite{CongreveWihler:15}, we formulate an adaptive procedure.
More specifically: as long as our adaptive porcedure is running, we either perform a fixed point iteration or refine the current space $X_{h}$ based on the derived error indicators.
In order to test such a fully adaptive procedure, we concentrate on semilinear elliptic boundary value problems with possible singular perturbations.

\subsubsection*{Linearization schemes}
Let us briefly address two possible strategies when solving nonlinear problems numerically. Firstly, the nonlinear problem can be formulated within an accurate finite dimensional framework. Based on a suitable iterative scheme, the resulting nonlinear finite dimensional problem will be linearized. Alternatively, a local linearization is applied. This leads to a sequence of linear problems which afterwards will be discretized by some suitable numerical approximation schemes. It is noteworthy that the second approach offers the application of the existing numerical analysis and the computational techniques for \emph{linear} problems (such as e.g. the development of classical residual-based error bounds). The concept of approximating infinite dimensional nonlinear problems by appropriate \emph{linear discretization schemes} has been studied by several authors in the recent past. For example, the approach presented in~\cite{CongreveWihler:15} (see also the work~\cite{GarauMorinZuppa:11,ChaillouSuri:07}) combines fixed point linearization methods and Galerkin approximations in the context of strictly monotone problems. Similarly, in~\cite{WihlerHouston:16,El-AlaouiErnVohralik:11,ErVo13,AmreinMelenkWihler:16,AmreinWihlerPseudo:16,AmreinWihler:15}, the nonlinear PDE problems at hand are linearized by an (adaptive) Newton technique, and subsequently discretized by a linear finite element method. On a related note, the discretization of a sequence of linearized problems resulting from the local approximation of semilinear evolutionary problems has been investigated in~\cite{AmreinWihlerTime:15}. In all of the works~\cite{AmreinMelenkWihler:16,AmreinWihlerPseudo:16	,AmreinWihler:15,AmreinWihlerTime:15,CongreveWihler:15}, the key idea in obtaining fully adaptive discretization schemes is to provide a suitable interplay between the underlying linearization procedure and (adaptive) Galerkin methods; this is based on investing computational time into whichever of these two aspects is currently dominant.

\subsubsection*{Outline}
The outline of this paper is as follows. In Section \ref{sec:errors} we study fixed point iterations within the context of general Hilbert spaces and derive an \emph{apriori} and \emph{a posteriori} error analysis. Subsequently, the purpose of Section~\ref{sc:Well-Posedness-FEM} is the discretization of the resulting sequence of {\em linear} problems by the finite element method and the development of an $\varepsilon$-robust \emph{a posteriori} error analysis. The final estimate (Theorem~\ref{thm:1}) bounds the error in terms of the (elementwise) finite element approximation (FEM-error) and the error caused by the fixed point iteration of the original problem. Then, in order to define a fully adaptive Fixed-Point-Galerkin scheme, we propose an interplay between the adaptive method and the adaptive finite element approach: More precisely, as the adaptive procedure is running, we either perform a fixed point iteration or refine the current finite element mesh based on the {\em a posteriori error} estimate (Section~\ref{sc:Well-Posedness-FEM}); this is carried out depending on which of the errors (FEM-error or fixed point error) is  more dominant in the present iteration step. In Section~\ref{sec:numerics} we provide a numerical experiment which shows that the proposed scheme is reliable and $\varepsilon$-robust for reasonable choices of initial guesses. Finally, we summarize and comment our findings in Section~\ref{sc:concl}.

\section{Apriori and a posteriori error estimates} 
\label{sec:errors}

First of all and with the purpose of completness, we recall the well known Banach's fixed point Theorem (see, e.g. \cite{Evans}), which asserts that any Lipschitz continuous map $\OB:X\rightarrow X$ with Lipschitz constant $\alpha \in (0,1) $, and operating over a Banach space $X$, possesses a unique fixed point. Furthermore, for a given initial guess $x_{0}\in X$, the fixed point can be obtained through iteration of $x_{n+1}=\OB(x_{n}) \quad n\geq 0$.  
Incidentally, there holds the following \emph{apriori} error estimate 
\begin{equation}
\label{eq:Banach-Apriori}
\norm{x-x_{n}}_{X}\leq \frac{\alpha^n}{1-\alpha}\norm{x_{0}-x_{1}}_{X}.
\end{equation}

We now show that the difference $u_{h}^{n}-u^{n}$ between the discrete sequence $\{u_{h}^{n}\}_{n\geq 0}$ and its continuous analogue $\{u^{n}\}_{n\geq 0}$ is bounded by the discretization error and the contraction constant $\alpha \in (0,1)$ . Indeed, there holds the following result:

\begin{theorem}
For given initial values $u_{h}^{0}= u^0\in X_{h}\subset X$ we consider the iterations 
\[
(u^{n+1},v)_{X}=\B(u^n)(v)\quad \forall v\in X, \qquad (u^{n+1}_{h},v_{h})_{X}=\B(u_{h}^n)(v_{h})\quad \forall v_{h}\in X_{h}.
\]
Assume that the disretization of $X$ is fine enough such that
\begin{equation}
\label{eq:fineness}
\abs{\B(u_{h}^n)(v)-(u_{h}^{n+1},v)_{X}}\leq \eta_{h}\norm{v}_{X}, \quad \forall v\in X.
\end{equation}

Then there holds 
\begin{equation}
\label{eq:distance}
\norm{u^{n+1}-u_{h}^{n+1}}_{X}\leq \frac{\eta_{h}}{1-\alpha} \quad \text{for all} \quad n\in \mathbb{N}.
\end{equation}
\end{theorem}

\begin{proof}
We proceed along the lines of~\cite[\S 8.1]{5}, where the authors show a similar result for Newton's method. 

The argument is inductive. Therefore, given a sequence $\{\varepsilon_{k}\}_{k\geq 0}\subset \mathbb{R}_{\geq 0}$ and assuming that there holds 
\begin{equation}\label{eq:verankerung}
\norm{u^{n}-u_{h}^{n}}_{X}\leq \varepsilon_{n}.
\end{equation}
Since $u^{0}=u_{h}^{0}$, we can choose $\varepsilon_{0}=0$ for $n=0$. For $n\geq 0$, we further introduce sequences $\{\tilde{u}^{k,n}\}_{k\geq n}$ generated by the iteration
\begin{equation}\label{eq:plugin}
(\tilde{u}^{k+1,n},v)_{X}=\B(\tilde{u}^{k,n})(v), \quad \tilde{u}^{n,n}:=u_{h}^{n}, \quad v\in X.
\end{equation}
We note that these sequences start at the discrete points $u_{h}^{n}$ and are convergent to the unique zero $u^{\infty}$ of $\F$ (see Figure~\ref{Bild_Sequence}).

\begin{figure}
\includegraphics[width=0.75\textwidth]{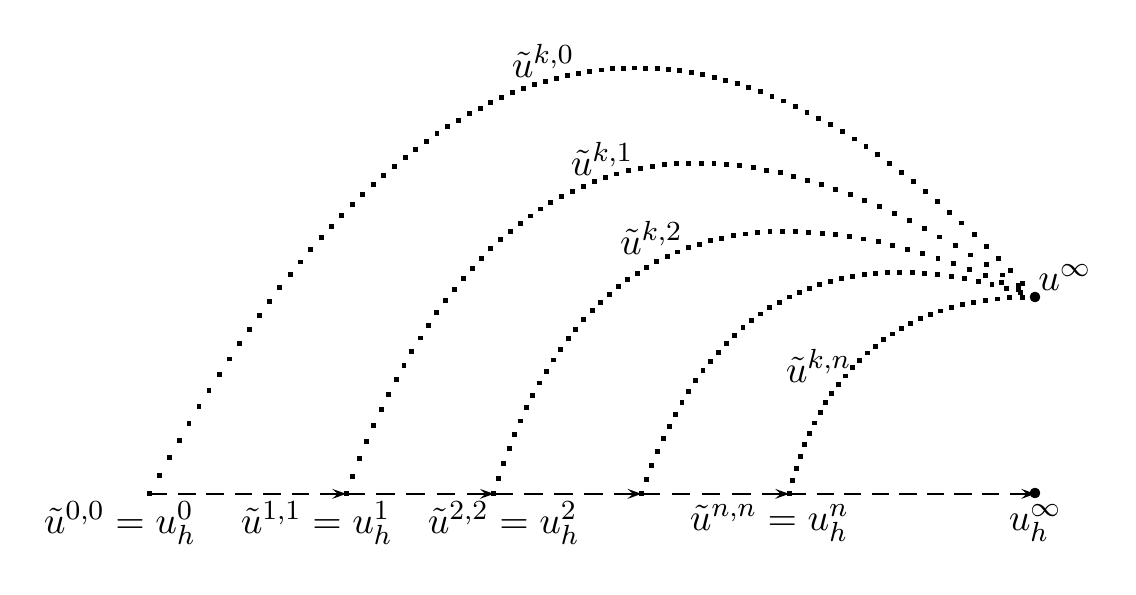}
\caption{The sequence given in \eqref{eq:plugin}.}
\label{Bild_Sequence}
\end{figure}

Induction step: employing the triangle inequality we get 
\begin{equation}
\label{eq:step1}	
\norm{u^{n+1}-u_{h}^{n+1}}_{X}\leq \norm{u^{n+1}-\tilde{u}^{n+1,n}}_{X}+\norm{\tilde{u}^{n+1,n}-u_{h}^{n+1}}_{X}.
\end{equation}
The first term can be estimated using \eqref{eq:verankerung} and the $\alpha$-Lipschitz continuity of $\B$:
\begin{align*}
\norm{u^{n+1}-\tilde{u}^{n+1,n}}_{X}^{2}&=(u^{n+1}-\tilde{u}^{n+1,n},u^{n+1}-\tilde{u}^{n+1,n})_{X}\\
&=(\B(u^{n})-\B(\tilde{u}^{n,n}))(u^{n+1}-\tilde{u}^{n+1,n})\\
&\leq \alpha \norm{u^{n}-\tilde{u}^{n,n}}_{X}\norm{u^{n+1}-\tilde{u}^{n+1,n}}_{X}\\
&= \alpha \norm{u^{n}-u_{h}^{n}}_{X}\norm{u^{n+1}-\tilde{u}^{n+1,n}}_{X}\\
&\leq  \alpha \varepsilon_{n}\norm{u^{n+1}-\tilde{u}^{n+1,n}}_{X},
\end{align*}
i.e., we have
\begin{equation}
\label{eq:epsilon}
\norm{u^{n+1}-\tilde{u}^{n+1,n}}_{X}\leq \alpha \varepsilon_{n}.
\end{equation}

For the second term $\norm{\tilde{u}^{n+1,n}-u_{h}^{n+1}}_{X}$ in \eqref{eq:step1} we use \eqref{eq:distance} and observe: 
\begin{equation}\label{eq:second-bound}
\begin{aligned}
\norm{\tilde{u}^{n+1,n}-u_{h}^{n+1}}_{X}^2=&(\tilde{u}^{n+1,n},\tilde{u}^{n+1}-u_{h}^{n+1})_{X}-\B(u_{h}^n)(\tilde{u}^{n+1,n}-u_{h}^{n+1})\\
&+\B(u_{h}^n)(\tilde{u}^{n+1,n}-u_{h}^{n+1})-(u_{h}^{n+1},\tilde{u}^{n+1,n}-u_{h}^{n+1})_{X}\\
\leq &\abs{\B(u_{h}^n)(\tilde{u}^{n+1,n}-u_{h}^{n+1})-(u_{h}^{n+1},\tilde{u}^{n+1,n}-u_{h}^{n+1})_{X}}\\
\leq &  \eta_{h}\norm{\tilde{u}^{n+1,n}-u_{h}^{n+1}}_{X},
\end{aligned}
\end{equation}
and therefore
\begin{equation}
\label{eq:discretization_error}
\norm{\tilde{u}^{n+1,n}-u_{h}^{n+1}}_{X}\leq \eta_{h}.
\end{equation}

Taking \eqref{eq:epsilon} into account we can bound \eqref{eq:step1} by
\[
\norm{u^{n+1}-u_{h}^{n+1}}_{X}\leq \alpha \varepsilon_{n}+\eta_{h}=:\varepsilon_{n+1}.
\]

We now consider the fixed point iteration (see Figure~\ref{Bild_iteration})
\begin{equation}
\label{eq:fixedpoint}
\varepsilon_{n+1}=\alpha \varepsilon_{n}+\eta_{h}, \quad \varepsilon_{0}=0.
\end{equation}

\begin{figure}
\includegraphics[width=0.53\textwidth]{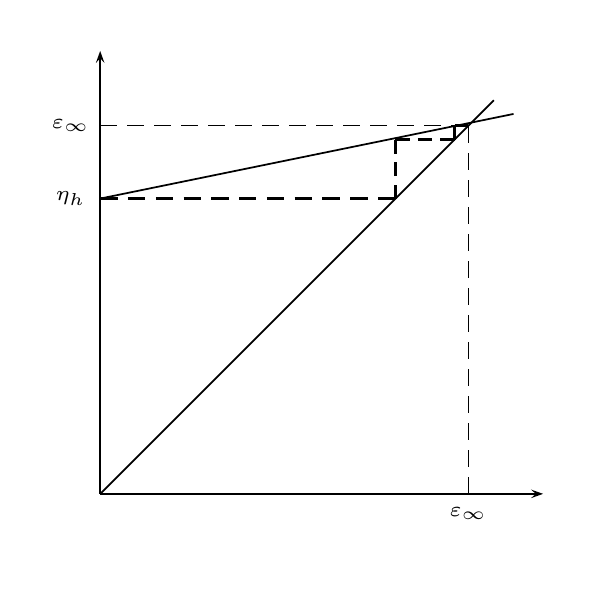}
\caption{The fixed point iteration given in \eqref{eq:fixedpoint}.}
\label{Bild_iteration}
\end{figure}

Since $\alpha \in (0,1)$ we conlude that $\varepsilon_{n}\leq \varepsilon_{\infty}=\frac{\eta_{h}}{1-\alpha}$ for all $n\geq 0$ which completes the inductive and therefore the whole argument.
\end{proof}

Based on this result, there holds the following \emph{apriori} bound:

\begin{cor}\label{cor2}
For any initial value $u^{0}\in X_{h}$ there holds the following \emph{apriori} error estimate:
\begin{equation}\label{eq:disc-apriori}
\norm{u-u_{h}^{n}}_{X}\leq \frac{1}{1-\alpha}\left(\alpha^{n}\norm{u^{1}-u^{0}}_{X}+\eta_{h}\right).
\end{equation}
\end{cor}
\begin{proof}
By virtue of the triangle inequality and employing \eqref{eq:Banach-Apriori} together with \eqref{eq:distance} we obtain
\[
\norm{u-u_{h}^{n}}_{X}\leq \norm{u-u^{n}}_{X}+\norm{u^{n}-u_{h}^{n}}_{X}\leq \frac{1}{1-\alpha}\left(\alpha^{n}\norm{u^{1}-u^{0}}_{X}+\eta_{h}\right).
\]
\end{proof}

Based on the Lipschitz continuity of $ \F$ we readily infer
\[
\norm{\F(u_{h}^n)}_{X'}=\norm{\F(u)-\F(u_{h}^n)}_{X'}\leq L\norm{u-u_{h}^n}_{X}.
\]
Thus we can bound the residual $\F(u_{h}^n)$ as follows:
\begin{cor}\label{cor1}
There holds 
\[
\norm{\F(u_{h}^n)}_{X'}\leq \frac{L}{1-\alpha}\left(\alpha^n\norm{u^1-u^0}_{X}+\eta_{h}\right), \quad \norm{\F(x)}_{X'}:=\sup_{\norm{v}_{X}=1}{\F(x)(v)}.
\]
\end{cor}

\subsection{A posteriori error analysis}
Following along the lines of \cite{CongreveWihler:15}, we now want to exploit an \emph{a posteriori} error analysis for strictly monotone operators. In doing so, we first recall the well known Theorem of Zarantonello (see \cite{12}): 
\begin{theorem}
Suppose that $\F:X\rightarrow X'$ is $L$-Lipschitz and strongly monotone, with monotonicity constant $c$. Then there exists a unique zero for $\F$.
\end{theorem}

\begin{proof}
Let $\J(x)(y)=(x,y)_{X} $ denote the Riesz-Isometry $\J:X\rightarrow X'$. We show that $\OB:X\rightarrow X$ given by $\OB(x):=\J^{-1}(\B(x))$ is $\alpha$-contractive and therefore possesses a unique fixed point $u$ which is the desired zero for the operator $ \OF(x):=\J^{-1}(\F(x))$. 

We have 
\begin{align*}
\norm{\OB(x)-\OB(y)}_{X}^{2}=\norm{x-y}_{X}^{2}-2t(x-y,\OF(x)-\OF(y))_{X}+t^2\norm{\OF(x)-\OF(y)}_{X}^2.
\end{align*}
Notice that
\[
\norm{\OF(x)-\OF(y)}_{X}=\norm{\F(x)-\F(y)}_{X'}\leq L\norm{x-y}_{X},
\]
and
\[
(\F(x)-\F(y))(x-y)=\J(\OF(x)-\OF(y))(x-y)=(\OF(x)-\OF(y),x-y)_{X}\geq c\norm{x-y}_{X}^{2}.
\]
Thus we arrive at 
\[
\norm{\OB(x)-\OB(y)}_{X}^{2}\leq f(t)\cdot \norm{x-y}_{X}^{2},  
\]	
with $f(t):=1-2ct+(Lt)^2$.
Moreover we observe that for $t\in (0,\nicefrac{2c}{L^2}) $ we have $\sqrt{f(t)}<1$. Note that the optimal contraction constant is 
$\alpha_{\text{opt}}:=f(t_{\text{opt}})=\sqrt{1-\nicefrac{c^2}{L^2}}$ with $ t_{\text{opt}}:=\frac{c}{L^2}$.
\end{proof}

Based on this result there holds the following \emph{a posteriori} error estimate:

\begin{proposition}\label{aposteriori} If $\F:X\rightarrow X'$ satisfies \eqref{eq:A1} and \eqref{eq:A2}, then for $ e_{h}^{n+1}:=u-u_{h}^{n+1}$, there holds the a posteriori error bound:
\begin{equation}\label{eq:aposteriori}	
\norm{e_{h}^{n+1}}_{X}\leq \frac{L^2}{c^2}\cdot \eta_h+
\frac{L}{c}\left(1+\frac{L}{c}\right)\cdot \norm{u_h^{n+1}-u_{h}^{n}}_{X}.
\end{equation}
\end{proposition}

\begin{proof}
We follow along the lines of \cite{CongreveWihler:15}. 
Recalling the coercivity of $\OF$ we get
\[
c\norm{e_{h}^{n+1}}_{X}^{2}\leq -(\OF(u_{h}^{n+1}),e_{h}^{n+1})_{X}.
\]
Next, we recall the sequences $\{\tilde{u}^{k,n}\}_{k\geq n}$ given in \eqref{eq:plugin}, i.e., we observe
\[
(\tilde{u}^{n+1,n},e_{h}^{n+1})_{X}=(\OB(\tilde{u}^{n,n}),e_{h}^{n+1})_{X}=(\OB(u_h^{n}),e_{h}^{n+1})_{X}=(u_{h}^{n},e_{h}^{n+1})_{X}-t(\OF(u_{h}^{n}),e_{h}^{n+1})_{X}.
\]
By virtue of the Cauchy-Schwarz inequality and the Lipschitz continuity of $\OF$ we get
\begin{align*}
&c\norm{e_{h}^{n+1}}_{X}^{2}\\
\leq & -(\OF(u_{h}^{n+1}),e_{h}^{n+1})_{X}+t^{-1}(\tilde{u}^{n+1,n}-u_{h}^{n},e_{h}^{n+1})_{X}+(\OF(u_{h}^{n}),e_{h}^{n+1})_{X}\\
=&t^{-1}(\tilde{u}^{n+1,n}-u_{h}^{n+1},e_{h}^{n+1})_{X}+t^{-1}(u_{h}^{n+1}-u_{h}^{n},e_{h}^{n+1})_{X}+(\OF(u_{h}^{n}),e_{h}^{n+1})_{X}-(\OF(u_{h}^{n+1}),e_{h}^{n+1})_{X}\\
\leq & t^{-1}\left(\norm{\tilde{u}^{n+1,n}-u_{h}^{n+1}}_{X}+\norm{u_{h}^{n+1}-u_{h}^{n}}_{X}\right)\norm{e_{h}^{n+1}}_{X}+L\norm{u_{h}^{n+1}-u_{h}^{n}}_{X}\norm{e_{h}^{n+1}}_{X}.
\end{align*}
Dividing by $c\norm{e_{h}^{n+1}}_{X}$ and using $t:=\frac{c}{L^2}$ we obtain
\[
\norm{e_{h}^{n+1}}_{X}\leq \frac{L^2}{c^2}\norm{\tilde{u}^{n+1,n}-u_{h}^{n+1}}_{X}+\frac{L}{c}\left(1+\frac{L}{c}\right)\norm{u_h^{n+1}-u_{h}^{n}}_{X}.
\]
Recalling \eqref{eq:discretization_error} we conclude \eqref{eq:aposteriori}.
\end{proof}

\section{Application to semilinear problems}
\label{sc:Well-Posedness-FEM}

\subsubsection{Problem formulation} In this section, we focus on the numerical approximation procedure for semilinear elliptic boundary value problems with possible singular perturbations. More precisely, for a fixed parameter $\varepsilon>0$ (possibly with $\varepsilon \ll 1$), and a continuous function $f:\mathbb{R}\rightarrow \mathbb{R}$, we consider the problem of finding $u:\Omega \rightarrow \mathbb{R}$ that satisfies
\begin{equation}\label{eq:semilinear}
\begin{aligned}
-\varepsilon \Delta u &= f(x,u), \quad &&\text{in} \  \Omega,\\
u&=0, \quad &&\text{on} \ 	\partial  \Omega.
\end{aligned}
\end{equation}

Here, $\Omega\subset\mathbb{R}^d$, with $d=1$ or $d=2$, is an open and bounded 1d interval or a 2d Lipschitz polygon, respectively. Problems of this type appear in a wide range of applications including, e.g., nonlinear reaction-diffusion in ecology and chemical models~\cite{CaCo03,Ed05,Fr08,Ni11,OkLe01}, economy~\cite{BaBu95}, or classical and quantum physics~\cite{BeLi83,St77}.

In this work, we are interested in a unique solution $u\in X:=H_{0}^{1}(\Omega) $ solving \eqref{eq:semilinear}: here, we denote by $H_{0}^{1}(\Omega)$ the standard Sobolev space of functions in $H^{1}(\Omega)=W^{1,2}(\Omega)$ with zero trace on $\partial \Omega$. Furthermore, the weak formulation of \eqref{eq:semilinear} is to find $u\in X:=H_{0}^{1}(\Omega)$ such that 
\[
\F_{\varepsilon}(u)(v)=0, \quad \forall v\in X,
\]
where 
\begin{equation}\label{eq:semilinearform}
\F_{\varepsilon}(u)(v):=\int_{\Omega}{\{\varepsilon \nabla u \nabla v-f(u)v\}\dx}.
\end{equation}
In addition, we introduce the inner product
\[
(u,v)_{X}:=\int_{\Omega}{\{uv+\varepsilon \nabla u \cdot \nabla v\}\dx},\qquad u,v\in X,
\]
with induced norm on $X$ given by
\[
\NN{u}_{\varepsilon,D}:=\Bigl(\varepsilon\norm{\nabla u}_{0,D}^2 +\norm{u}_{0,D}^2 \Bigr)^{\nicefrac{1}{2}},\qquad u\in H^1(D),
\]
where $\|\cdot\|_{0,D}$ denotes the $L^2$-norm on~$D$. Frequently, for~$D=\Omega$, the subindex~`$D$' will be omitted.
Note that in the case of $ f(u)=-u+g$, with $g \in L^{2}(\Omega)$, i.e., when \eqref{eq:semilinear} is linear and strongly elliptic, the norm $ \NN{\cdot}_{\varepsilon,\Omega}$ is a natural energy norm on $X$.

In what follows we shall use the abbreviation $ x\preccurlyeq y $ to mean $x\leq cy $, for a constant $c>0$ independent of the mesh size $h$ and of~$ \varepsilon>0$.

\subsubsection{Fixed-Point-Iteration} For $u^n\in X$ and $t>0$, the fixed-point iteration is to find $u^{n+1}\in X$ from $u^n$ such that
\begin{equation}\label{eq:1}
(u^{n+1},v)_{X}=\B_{\varepsilon}(u^n)(v), \qquad \forall v\in X,
\end{equation} 
where for fixed $u\in X$, $t>0$, we set
\begin{equation}
\label{eq:FixedPoint}
\B_{\varepsilon}(u)(v):=(u,v)_{X}-t\int_{\Omega}{\left\{\varepsilon \nabla u \cdot \nabla v-f(u)v\right\}\dx}.
\end{equation}

\begin{remark}
We assume that $f$ is Lipschitz continuous with Lipschitz constant $L_{f}$. Furthermore we need the following monotonicity property:
\begin{equation}
\label{eq:monoton}
(f(x)-f(y))(x-y)\leq -c_{f} (x-y)^2, \quad c_{f}>0.
\end{equation}
As a consequence, the operator $\F$ in \eqref{eq:semilinearform} is Lipschitz continuous and strongly monotone with Lipschitz constant $L=\max(1,L_{f})$ and monotonincity constant $c=\min(1,c_{f})$ respectively. In particular, applying the iteration given in \eqref{eq:1}, we obtain a unique fixed point $x\in X$ for $ \B_{\varepsilon}$ which is the unique root of $\F_{\varepsilon}$ given in \eqref{eq:semilinearform}.
\end{remark}

\subsection{Galerkin Discretization}
In order to provide a numerical approximation of~\eqref{eq:semilinear}, we will discretize the \emph{linear} weak formulation~\eqref{eq:1} by means of a finite element method, which constitutes a Fixed-Point-Galerkin approximation scheme. Furthermore, we shall derive {\em a posteriori} error estimates for the finite element 
discretization which allow for an adaptive refinement of the meshes in each iteration step. 
This, together with the \emph{a posteriori} error estimate from Proposition~\ref{aposteriori} leads to a fully adaptive Fixed-Point-Galerkin discretization method 
for~\eqref{eq:1}.

\subsubsection{Finite Element Meshes and Spaces}
Let $ \mathcal{T}^h=\{T\}_{T\in\mathcal{T}^h}$ be a regular and shape-regular mesh partition of $\Omega $ into disjoint open simplices, i.e., any~$T\in\mathcal{T}^h$ is an affine image of the (open) reference simplex~$\widehat T=\{\widehat x\in\mathbb{R}_+^d:\,\sum_{i=1}^d\widehat x_i<1\}$. By~$h_T=\mathrm{diam}(T)$ we signify the element diameter of~$T\in\mathcal{T}^h$, and by $h=\max_{T\in\mathcal{T}^h}h_T$ the mesh size. Furthermore, by $\mathcal{E}^h$ we denote the set of all interior mesh nodes for~$d=1$ and interior (open) edges for~$d=2$ in~$\mathcal{T}^h$. In addition, for~$T\in\mathcal{T}^h$, we let~$\mathcal{E}^h(T)=\{E\in\mathcal{E}^h:\,E\subset\partial T\}$. For~$E\in\mathcal{E}^h$, we let~$h_E$ be the mean of the lengths of the adjacent elements in 1d, and the length of~$E$ in~2d. Let us also define the following two quantities:
\begin{equation}\label{boundary}
\begin{aligned}
\alpha_T&:=\min(1,\varepsilon^{-\nicefrac12}h_T),\qquad
\alpha_E:=\min(1,\varepsilon^{-\nicefrac12}h_E),
\end{aligned}
\end{equation}
for~$T\in\mathcal{T}^h$ and~$E\in\mathcal{E}^h$, respectively.

We consider the finite element space of continuous, piecewise linear functions on $\mathcal{T}^h$ with zero trace on~$\partial\Omega$, given by
\begin{equation*}
V_{0}^{h}:=\{\varphi\in H^1_0(\Omega):\,\varphi|_{T} \in \mathbb{P}_{1}(T) \, \forall T \in \mathcal{T}^h\},
\end{equation*}
respectively, where~$\mathbb{P}_1(T)$ is the standard space of all linear polynomial functions on~$T$.

\subsubsection{Linear Finite Element Discretization}
For $t=\frac{c}{L^2}$ and $u_{h}^n \in V_{0}^h, n\geq 0$, we consider the finite element approximation of \eqref{eq:1}, which is to find $ u_{h}^{n+1}\in V_{0}^{h}$ such that 
\begin{equation}\label{eq:2}
(u_{h}^{n+1},v_{h})_{X}=\B_{\varepsilon}(u_{h}^n)(v_{h}), \quad \forall v_{h}\in V_{0}^h,
\end{equation}
where, for a fixed $u_{h}\in V_{0}^{h}$, 
\[
\B_{\varepsilon}(u_{h})(v_{h})=(u_{h},v_{h})_{X}-t\int_{\Omega}{\left\{\varepsilon \nabla u_{h} \cdot \nabla v_{h}-f_{h}(u_{h})v_{h}\right\}}\dx.
\]
Here, $f_{h}(u_{h})\in V_{h}$ is defined through
\[
\int_{\Omega}{(f(u_{h})-f_{h}(u_{h}))v_{h}\dx}=0,\quad \forall v_{h}\in V_{0}^h.
\]
More precisely, if $V_{0}^{h}$ is spanned by the basis functions $\{\phi_{i}\}_{i=1}^{N}$ we solve the algebraic system
\begin{equation}
\label{eq:algebraic}
\sum_{k =1}^{N}{B_{ki}u_{k}^{n+1}}=\sum_{k =1}^{N}{B_{ki}u_{k}^{n}}-t\varepsilon \sum_{k =1}^{N}{A_{ki}u_{k}^{n}}+t b(u_{h}^{n})_{i}, \quad i\in\{1,\ldots,N\},
\end{equation}
with respect to $\{u_{k}^{n+1}\}_{k=1}^{N}$ and set $u_{h}^{n+1}=\sum_{k=1}^{N}{u_{k}^{n+1}\phi_{k}}$. 
In \eqref{eq:algebraic}, $B$ signifies the iteration matrix given by $B_{ik}=(\phi_{i},\phi_{k})_{X}$ and $A$ denotes the stiffness matrix. Moreover, for $ i=\{1,\ldots,N\} $ the load vector is given by $b(u_{h}^n)_{i}:=\int_{\Omega}{f(u_{h}^{n})\phi_{i}\dx}$.

\subsection{A Posteriori Analysis}
The aim of this section is to derive \emph{a posteriori} error bounds for the FEM iteration \eqref{eq:2}. In view of Proposition~\ref{aposteriori} it is sufficient to derive a computable quantity $\eta(u_{h}^{n+1},V_{0}^h) $.

Therefore we introduce the quantity:
\[
\tilde{e}_{h}^{n+1,n}:=\tilde{u}^{n+1,n}-u_{h}^{n+1}.
\]
Moreover, let $\I:\,H_{0}^{1}(\Omega)\rightarrow V_{0}^{h} $ be the quasi-interpolation operator of Cl\'ement (see, e.g., \cite[Corollary~4.2]{AmreinWihler:15}) and set $v_{h}:=\I \tilde{e}_{h}^{n+1,n} $.

We observe
\begin{align*}
\NN{\tilde{e}_{h}^{n+1,n}}_{\varepsilon}^{2}&=(\tilde{u}^{n+1,n},\tilde{e}_{h}^{n+1,n})_{X}-(u_{h}^{n+1},\tilde{e}_{h}^{n+1,n})_{X}\\
&=(\tilde{u}^{n+1,n},\tilde{e}_{h}^{n+1,n}-v_{h})_{X}-(u_{h}^{n+1},\tilde{e}_{h}^{n+1,n}-v_{h})_{X}\\
&=-t\int_{\Omega}{\left\{\varepsilon \nabla u_{h}^{n}\nabla(\tilde{e}_{h}^{n+1,n}-v_{h})-f(u_{h}^n)(\tilde{e}_{h}^{n+1,n}-v_{h})\right\}\dx}-(u_{h}^{n+1}-u_{h}^{n},\tilde{e}_{h}^{n+1,n}-v_{h})_{X}\\
&=-\sum_{T\in\T_{h}}\int_{T}{\left\{\varepsilon \nabla (u_{h}^{n+1}-u_{h}^{n}) \nabla (\tilde{e}_{h}^{n+1}-v_{h})+t\varepsilon \nabla u_{h}^{n}\nabla (\tilde{e}_{h}^{n+1,n}-v_{h})\right\}\dx}\\
&\quad +\sum_{T\in\T_{h}}\int_{T}{\left\{tf_{h}(u_{h}^{n})(\tilde{e}_{h}^{n+1,n}-v_{h})-(u_{h}^{n+1}-u_{h}^{n})(\tilde{e}_{h}^{n+1,n}-v_{h})\right\} \dx}\\
&\quad +\sum_{T\in\T_{h}}\int_{T}{t(f(u_{h}^{n})-f_{h}(u_{h}^{n}))(\tilde{e}_{h}^{n+1,n}-v_{h})\dx.} 
\end{align*}

Integrating by parts in the first term on the right-hand side, recalling the fact that~$(v-\I v)=0$ on~$\partial\Omega$, and applying some elementary calculations, yields that

\[
\NN{\tilde{e}_{h}^{n+1,n}}_{\varepsilon}^{2}=\sum_{T\in\T_{h}}{(b_{T}+c_T)}+\sum_{E\in\E_{h}}{a_{E}}
\]
where
\begin{align*}
b_{T}&:=\int_{T}{\left\{(\varepsilon \Delta (u_{h}^{n+1}-u_{h}^{n})-(u_{h}^{n+1}-u_{h}^{n})+t\varepsilon \Delta u_{h}^{n} +tf_h(u_{h}^{n}))(\tilde{e}_{h}^{n+1,n}-v_{h})\right\}\dx},\\
a_{E}&:=\int_{\partial E}{\varepsilon \jmp{\nabla (u_{h}^{n+1}-u_{h}^{n})+t\nabla u_{h}^{n}}(\tilde{e}_{h}^{n+1,n}-v_{h})\ds},\\
c_{T}&:=\int_{T}{t(f(u_{h}^{n})-f_{h}(u_{h}^{n}))(\tilde{e}_{h}^{n+1,n}-v_{h})\dx.} 
\end{align*}

with~$E\in\mathcal{E}^h$, $T\in\mathcal{T}^h$. Here, for any edge $ E=\partial T^\sharp\cap \partial T^\flat \in \mathcal{E}^h $ shared by two neighboring elements~$T^\sharp, T^\flat\in\mathcal{T}^h$, where $\bm n^\sharp$ and~$\bm n^\flat$ signify the unit outward vectors on~$\partial T^\sharp$ and~$\partial T^\flat$, respectively, we denote by 
\[
\jmp{\nabla u_{h}^{n+1}}(\bm x)=\lim_{t\to 0^+}\nabla u_h^{n+1}(\bm x+t\bm n^\sharp)\cdot\bm n^\sharp+\lim_{t\to 0^+}\nabla u_h^{n+1}(\bm x+t\bm n^\flat)\cdot\bm n^\flat,\qquad \bm x\in E,
\]
the jump across~$E$.

Then, for $T\in \mathcal{T}^{h}$, defining the FEM-error
\begin{equation}
\label{Femerror}
\begin{aligned}
\eta_{n+1,T}^2&:= \alpha_{T}^2 \norm{\varepsilon \Delta (u_{h}^{n+1}-u_{h}^{n})-(u_{h}^{n+1}-u_{h}^{n})+t\varepsilon \Delta u_{h}^{n} +tf_h(u_{h}^{n})}_{0,T}^2\\
&\quad+\frac{1}{2}\sum_{E\in \mathcal{E}^{h}(T)}{\varepsilon^{-\nicefrac12}\alpha_E\norm{\varepsilon \jmp{\nabla (u_{h}^{n+1}-u_{h}^{n})+t\nabla u_{h}^{n}}}_{0,E}^2}.
\end{aligned}
\end{equation}
with~$\alpha_T$ and~$\alpha_E$ from~\eqref{boundary}, we proceed along the lines of the proof of~\cite[Theorem~4.4]{AmreinWihler:15} in order to obtain the following result.

\begin{proposition}
\label{FEM}
For $n\geq 0$ there holds the upper a posteriori bound
\begin{equation}\label{eq:help}
\NN{\tilde{e}_{h}^{n+1,n}}_{\varepsilon}\preccurlyeq \left(t\norm{f(u_{h}^n)-f_{h}(u_{h}^n)}_{0,\Omega}^2+\sum_{T\in\T_{h}}{\eta_{n+1,T}^2}\right)^{\nicefrac{1}{2}}=:\eta(u_{h}^{n+1},V_{0}^h)
\end{equation}
with $ \eta_{n+1,T}$ , $ T\in \mathcal{T}^{h}$ from \eqref{Femerror}.
\end{proposition}

Using Proposition~\ref{aposteriori} together with the bound \eqref{eq:help}, we end up with the following \emph{a posteriori} error bound:
\begin{theorem}
\label{thm:1}
For $n\geq 0$ there holds the upper a posteriori error bound 
\begin{equation}\label{eq:goal}
\NN{e_{h}^{n+1}}_{\varepsilon}\preccurlyeq \eta(u_{h}^{n+1},V_{0}^h)+\eta_{\mathrm{FP}}(u_{h}^{n+1},u_{h}^{n},V_{0}^h),
\end{equation}
with 
\begin{equation}\label{eq:linearizationerror}
\eta_{\mathrm{FP}}(u_{h}^{n+1},u_{h}^{n},V_{0}^h):=\NN{u_{h}^{n+1}-u_{h}^n}_{\varepsilon},
\end{equation} 
$ T\in\mathcal{T}^h $ and $ \eta(u_{h}^{n+1},V_{0}^h) $, from \eqref{eq:help}, respectively. 
\end{theorem}

\subsection{A Fully Adaptive Fixed-Point-Galerkin Algorithm}
We will now propose a procedure that will combine the fixed point iteration procedure with an automatic finite element mesh refinement strategy. More precisely, based on the \emph{a posteriori} error bound from Theorem~\ref{thm:1}, the main idea of our approach is to provide an interplay between the fixed point iterations and adaptive mesh refinements which is based on monitoring the error indicators in~\eqref{eq:linearizationerror} and~\eqref{eq:help}, and on acting according to whatever quantity is dominant in the current computations. 
 
The individual computational steps are summarized in Algorithm~\ref{al:full}.

\begin{algorithm}
\caption{Fully-adaptive Fixed-Point-Galerkin method}
\label{al:full}
\begin{algorithmic}[1]
\State Initialization: Input an initial starting space $V_0^h$, a lower bound $h_{\text{min}}>0$ for the fineness parameters $h$, an initial guess $u_{h}^{0}\in V_0^h$, and a refinement parameter $\theta>0$. 
\While {$h\geq h_{\text{min}}$} \Comment{control the degrees of freedom}
\myState {On the current space~$ V_{0}^h $, compute $ u_{h}^{n+1} $, and evaluate the error indicators $\eta(u_{h}^{n+1},V_{0}^h)$ and $\eta_{\mathrm{FP}}(u_{h}^{n+1},u_{h}^n,V_{0}^h)$ from~\eqref{eq:help}, and \eqref{eq:linearizationerror}.}

\If {$\eta_{\mathrm{FP}}(u_{h}^{n+1},u_{h}^n,V_{0}^h)< \theta \cdot  (\eta(u_{h}^{n+1},V_{0}^h)$}  
\myStateTriple{Refine the mesh $\mathcal{T}^h$ adaptively based on the element wise error indicators $\eta_{n+1,T}$, $T\in \mathcal{T}^h$, from Proposition \ref{FEM}, obtain a new mesh $\mathcal{T}^{h}\subset\mathcal{T}^{\tilde{h}}$, set $h:=\tilde{h}$ and go back to step $(2)$ with the previously computed solution $u_{h}^{n+1}$ as interpolated on the refined mesh $\mathcal{T}^{\tilde{h}}$.}

\Else 
\myStateTriple {do another fixed point iteration by going back to step $(3)$.}
\EndIf
\myState{set~$n\leftarrow n+1$.}
\EndWhile
\end{algorithmic}
\end{algorithm}

\section{Numerical Experiment}
\label{sec:numerics}
We will now illustrate and test the above fully adaptive Algorithm~\ref{al:full} with a numerical experiment in 2d. The linear systems resulting from the finite element discretization~\eqref{eq:algebraic} are solved by means of a direct solver.

\begin{figure}
\includegraphics[width=0.43\textwidth]{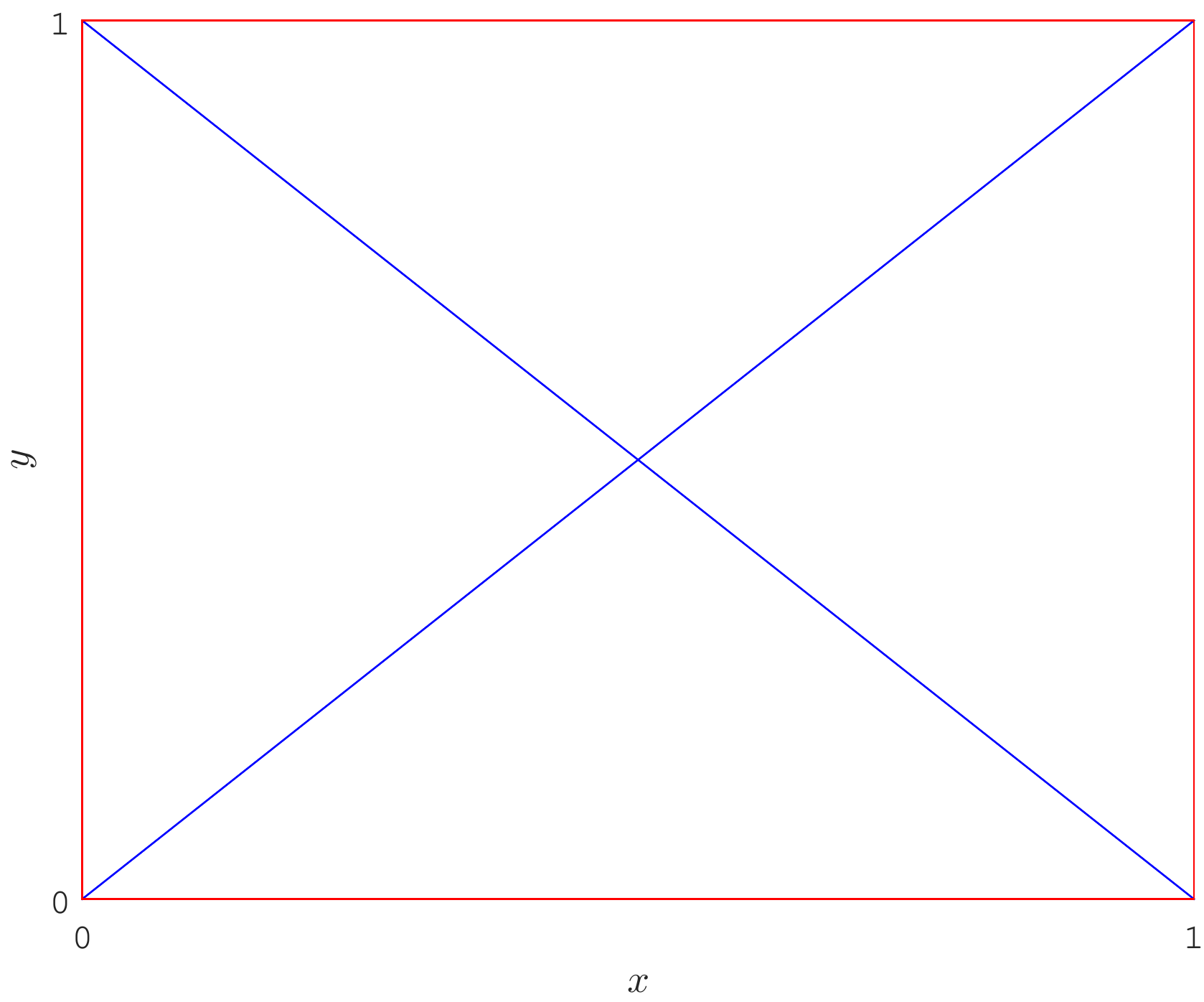}
\hfill
\includegraphics[width=0.43\textwidth]{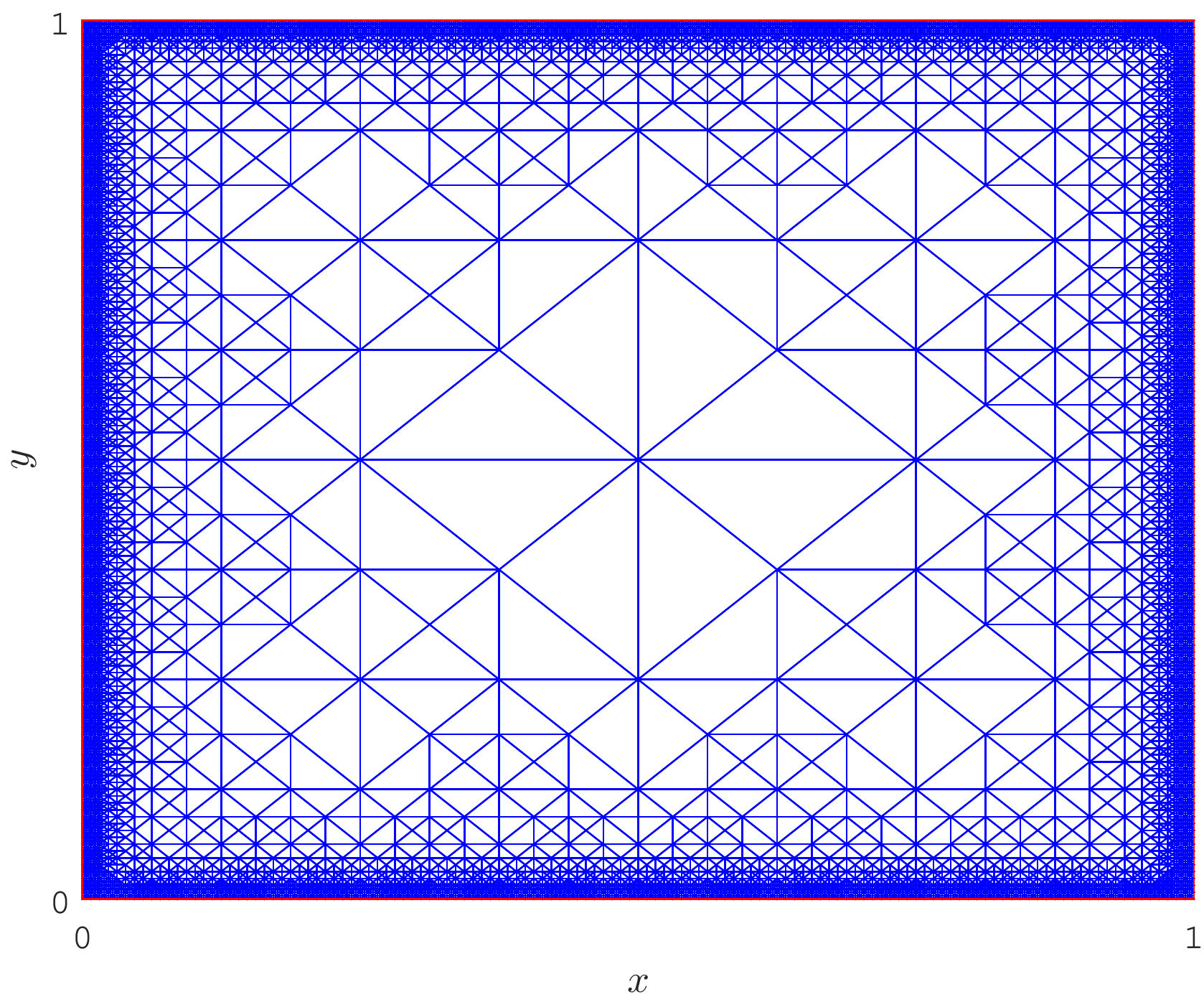}
\caption{Example~\ref{ex:2} for $\varepsilon =10^{-7}$: Initial mesh for the numerical solution (left), and the mesh corresponding to the numerical solution (right).}	
\label{Example2}
\end{figure}

\begin{example}\label{ex:2}
We consider the problem
\begin{equation}\label{eq:problem2}
-\varepsilon \Delta u = f(u), \quad \text{in} \quad \Omega=(0,1)^2, \quad u = 0 \quad \text{on} \quad  \partial \Omega,
\end{equation}
with $ f(u)=\frac{1-u}{1+\mathrm{e}^{-(u-1)^2}}$. Here $ \abs{\partial_{u}f(u)}$ is uniformely bounded (roughly by $1.3$) and 
\[\partial_{u}f(u)\leq \partial_{u} f(u)|_{u=1}=-\nicefrac{1}{2}.\] Henceforth we have $L=L_{f}\approx 1.3$, $c=c_{f}=\nicefrac{1}{2}$, i.e., the problem is well defined. 

The focus of this experiment is on the robustness of the \emph{a posteriori} error bound \eqref{thm:1} with respect to the singular perturbation parameter $\varepsilon$ as 
$\varepsilon \to 0$.
Indeed, neglecting the boundary conditions for a moment, one observes that the unique positive zero $ u = 1$ of $ f(u) $ is a solution of the PDE. We therefore expect boundary layers along $\partial \Omega$; see Figures \ref{Example2} and \ref{Example2perf} (right).

Starting from the initial mesh depicted in Figure \ref{Example2} (left) with $u_{0}^{h}(\nicefrac{1}{2},\nicefrac{1}{2})= 1$, we test the fully adaptive fixed point Galerkin Algorithm \ref{al:full} for different choices of $\varepsilon=\{10^{-i}\}_{i=0}^{8}$. 
In Algorithm \ref{al:full} the parameter $\theta$ is chosen to be $0.5$. Furthermore, in this example, the procedure is always initiated with a uniform initial mesh $\mathcal{T}$ consisting of $4$ elements; see Figure \ref{Example2} (left). As $\varepsilon\to 0 $ the resulting solutions feature ever stronger boundary layers; see Figures \ref{Example2} and \ref{Example2perf} \ (right). Again we see that the performance data in Figure~ 
\ref{Example2perf} shows errors decay, firstly, robust in $\varepsilon$, and, secondly, of (optimal) order~$\nicefrac{1}{2}$ with respect to the number of degrees of freedom.

\end{example}

\begin{figure}[h!]
\includegraphics[width=0.43\textwidth]{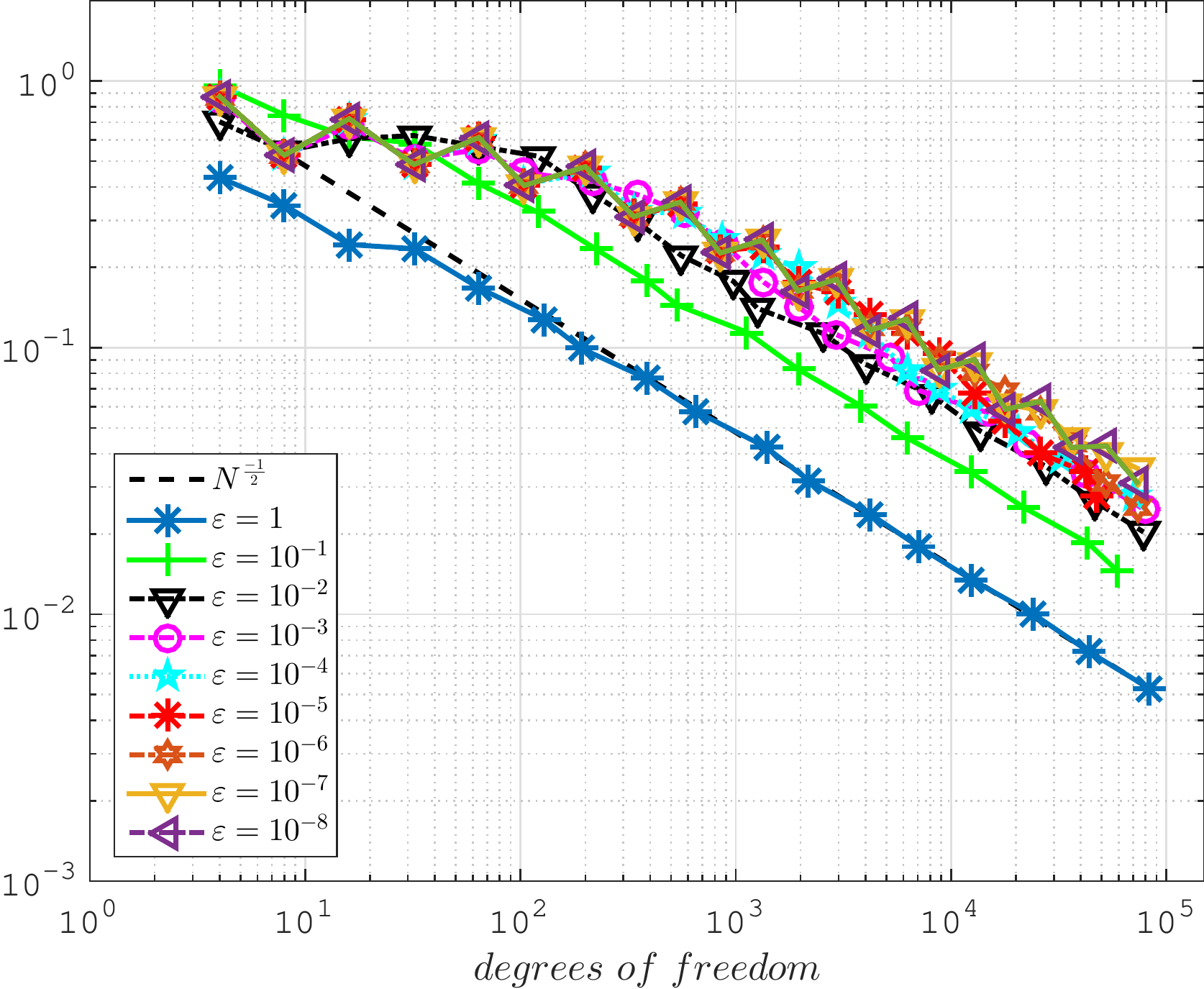}
\hfill
\includegraphics[width=0.55\textwidth]{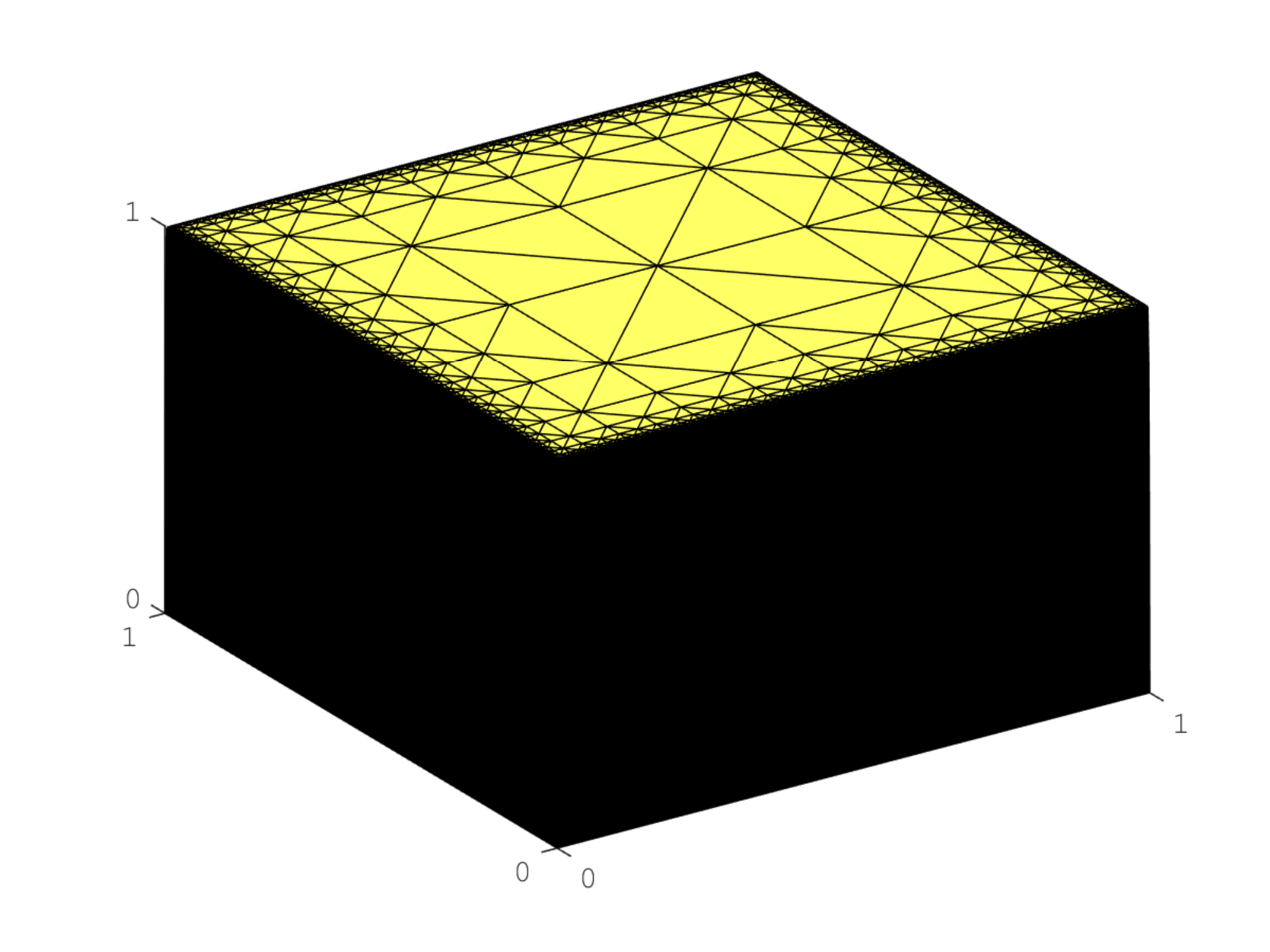}
\caption{Example~\ref{ex:2}: Estimated errors for different choices of $\varepsilon$ (left) and the numerical solution of \ref{eq:problem2} for $\varepsilon \ll 1$ (right).}	
\label{Example2perf}
\end{figure}

\section{Conclusions}
\label{sc:concl}
The aim of this paper is to introduce a reliable and computationally feasible procedure for the numerical solution of semilinear elliptic boundary value problems, with possible singular perturbations. The key idea is to combine a simple fixed point method with an automatic mesh refinement finite element procedure. Furthermore, the sequence of linear problems resulting from the application of the fixed point iteration and Galerkin discretization is treated by means of a robust (with respect to the singular perturbations) {\em a posteriori} error analysis and a corresponding adaptive mesh refinement process. Our numerical experiments clearly illustrate the ability of our approach to reliably find the solution of the underlying well posed problem, and to robustly resolve the singular perturbations at an optimal rate.

\bibliographystyle{amsplain}
\bibliography{FixedPoint-references}
\end{document}